\documentclass[12pt]{amsart}
\usepackage{amsmath,amssymb,amsfonts}
\usepackage[all]{xy}
\usepackage[T5,T1]{fontenc}
\usepackage{hyperref}
\usepackage{a4wide}

\theoremstyle{plain}
\newtheorem{thm}{Theorem}[section]
\newtheorem{lem}[thm]{Lemma}
\newtheorem{cor}[thm]{Corollary}
\newtheorem{prop}[thm]{Proposition}

\theoremstyle{definition}
\newtheorem{defn}[thm]{Definition}

\newtheorem{question}[thm]{Question}

\newtheorem{rmk}[thm]{Remark}
\newtheorem{rmks}[thm]{Remarks}
\newcommand{\ed}{{\rm ed}}

\newcommand{\Hom}{{\rm Hom}}
\newcommand{\im}{{\rm im}}
\newcommand{\Spec}{{\rm Spec \,}}

\newcommand{\Char}{{\rm char}}

\newcommand{\trdeg}{{\rm trdeg}}
\newcommand{\Gal}{{\rm Gal}}

\newcommand{\sF}{{\mathcal F}}

\newcommand{\sP}{{\mathcal P}}

\newcommand{\sR}{{\mathcal R}}

\newcommand{\A}{{\mathbb A}}

\newcommand{\F}{{\mathbb F}}
\newcommand{\G}{{\mathbb G}}
\renewcommand{\H}{{\mathbb H}}

\newcommand{\Z}{{\mathbb Z}}

\def\NDT{{\fontencoding{T5}\selectfont Nguy\~ \ecircumflex n Duy T\^an}}
\def\NQT{{\fontencoding{T5}\selectfont Nguy\~ \ecircumflex n Qu\'\ocircumflex{}c Th\'\abreve{}ng }}

\begin{document}
\title[Essential dimension of unipotent groups]{On the essential dimension of  unipotent algebraic groups} 
 \author{ \NDT }
 \address{ Universit\"at Duisburg-Essen, FB6, Mathematik, 45117 Essen, Germany and Institute of Mathematics, 18 Hoang Quoc Viet, 10307, Hanoi - Vietnam } 
\email{duy-tan.nguyen@uni-due.de}
\thanks{Partially supported by the NAFOSTED, the SFB/TR45 and the ERC/Advanced Grant 226257}

\begin{abstract}
We give an upper bound for the essential dimension of a smooth unipotent algebraic group over an arbitrary field. We also show that over a field $k$ which is finitely generated over a perfect field, a smooth unipotent  algebraic $k$-group is of essential dimension 0 if and only if it is $k$-split.

\end{abstract}

\maketitle

\section{Introduction}
Let $k$ be a base field, ${\rm Fields}_k$ the category of field extensions $K/k$, ${\rm Sets}$ the category of sets. Let $\sF : {\rm Fields}_k \to {\rm Sets}$ be a covariant functor. 
Given a field extension $K/k$, we will say that $a \in \sF(K)$ {\it descends} to an intermediate field $k \subset K_0 \subset K$ if $a$ is in the image of the induced map $F(K_0) \to F(K)$. 
The {\it essential dimension} $\ed_k(a)$ of $a \in \sF(K)$ is the minimum of the transcendence degrees ${\rm trdeg}_k(K_0)$ taken over all fields $k \subset K_0 \subset K$ such that $a$ descends to $K_0$. 
The essential dimension $\ed_k(\sF)$ of the functor $\sF$ is the supremum of $\ed(a)$ taken over all $a \in \sF(K)$ with $K$ in ${\rm Fields}_k$.

If $G$ is an algebraic group over $k$, we write $\ed_k(G)$  for the essential dimension of the functor $K\mapsto H^1_{\rm fppf}(K,G)$. The notion of essential dimension of a finite group is introduced by Buhler and Reichstein (\cite{BR}). 
The definition of the essential dimension of a functor is a generalization given later by Merkujev (\cite{BF}). 
In \cite{BRV1}, the authors introduce a notion of essential dimension for algebraic stack, see also \cite{BRV2}. Nowadays, studying essential dimension is an active area. See \cite{Re} and reference therein.

Computing  the essential dimension of algebraic groups is, in general, a hard problem. By the work of \cite{Fl, KM}, one  now can compute the essential dimension of finite (abstract) $p$-groups over a field of characteristic different from $p$. 
In \cite{LMMR}, the authors study also the essential dimension of algebraic tori. However, we do not know much about the essential dimension of finite $p$-groups over a field of characteristic $p>0$ in particular, and the essential dimension of unipotent algebraic groups in general. 
Let $k$ be a field of characteristic $p>0$ and $G$ be a finite $p$-group of order $p^n$. Then, Ledet \cite{Le} shows that $\ed_k(G)\leq n$. He also conjectures that $\ed_k(\Z/p^nZ)=n$. 
As noted by Reichstein \cite[Subsection 7.3]{Re}: This seems to be out of reach at the moment, at least for $n\geq 5$. 
Tossici and Vistoli \cite{TV} shows also  that the above inequality, $\ed_k(G)\leq n$, still holds for any finite (not necessarily smooth) trigonalizable $k$-group scheme $G$ of order $p^n$, where $p=\Char k$.

In this paper, we study the essential dimension of a unipotent algebraic group over a field. An {\it algebraic group} over a field $k$ is a $k$-group scheme of finite type over $k$. 
The smooth affine algebraic $k$-groups considered here are the same as linear algebraic groups defined over $k$ in the sense of \cite{Bo}.  
Recall that an affine algebraic $k$-group $G$ is called {\it unipotent} if $G_{\bar{k}}$ (the base change of $G$ to a fixed algebraic closure $\bar{k}$ of $k$) admits a finite composition series over $\bar{k}$ with each successive quotient isomorphic to a $\bar{k}$-subgroup of the additive group $\G_a$.
It is well-known  that an affine algebraic $k$-group $G$ is unipotent if and only if is $k$-isomorphic to a closed $k$-subgroup scheme of the group $T_n$ consisting of upper triangular matrices of order $n$ with all 1 on the diagonals, for some $n$. 

A smooth unipotent algebraic group $G$ over a field $k$ is called $k$-{\it split} if it admits a composition series by $k$-subgroups with successive quotients are $k$-isomorphic to the additive group $\G_a$. 
%We say that $G$ is $k$-{\it wound} if any $k$-homomorphism of algebraic $k$-groups $\G_a\to G$ is trivial. 
We say that $G$ is $k$-{\it wound} if every map of $k$-scheme $\A^1_k\to G$ is a constant map to a point in $G(k)$.

For any smooth unipotent algebraic group $G$ defined over $k$, there is a maximal $k$-split $k$-subgroup $G_s$, and it enjoys the following properties: it is normal in $G$, the quotient $G/G_s$ is $k$-wound and the formation of $G_s$ commutes with separable (not necessarily algebraic) extensions, see \cite[Chapter V, 7]{Oe} and \cite[Theorem B.3.4]{CGP}. 
The group $G_s$ is called the $k$-{\it split part} of $G$. We obtain the following result. 
\begin{thm}
\label{thm:main}
 Let $G$ be a smooth unipotent algebraic group over a field $k$,  $G_s$ its $k$-split part and let $H$  be the quotient $G/G_s$. Let $H^0$ be the identity component of $H$. Let $p^n$ be the order of $H/H^0$ if $p=\Char(k)>0$  and let $n=0$ if $\Char(k)=0$.  Then 
\[\ed_k(G)\leq \ed_k(H/H_0)+\dim (G/G_s)\leq n+\dim (G/G_s).\]
\end{thm}

In Section 2, we prove a technical result, Proposition \ref{prop:tech}, which is needed in proving Theorem \ref{thm:main}. In \cite[Lemma 3.4]{TV}, the authors  prove the proposition %\ref{prop:tech}
 for (not necessarily smooth) affine group schemes but under the assumption that $A$ is a \emph{commutative} unipotent normal subgroup scheme of $B$ (notations as in Proposition \ref{prop:tech}). 
In fact, they need the commutativity property of $A$ in their proof. 
Since  all groups considered in  Proposition \ref{prop:tech} are supposed to be smooth, we can use the language of cocycles and non-abelian cohomology theory as developed in \cite{Se2} and we can relax the commutativity condition on $A$.
%And we do not need to assume that $A$ is commutative.
%They use the language of torsors and non-abelian cohomology theory as developed in \cite{Gir}. In our case, we use the language of cocycles and non-abelian cohomology theory as developed in \cite{Se2} and we do not need  to assume that $A$ is commutative. 
%We believe that using the language of cocycles in flat cohomology as in \cite{Ga}, our proof of Proposition \ref{prop:tech} work also well for affine group schemes, i,e,, the statement of \cite[Lemma 3.4]{TV} still holds true without the condition of commutativity on $G_1$ (notation as in {\it loc.cit.}). 

In Section 3, we give some results concerning the essential dimension of  finite \'etale group schemes of $p$-power order  over fields of characteristic $p>0$. Some of the results are already appeared in \cite{JLY} in the case of finite abstract $p$-groups.

In this Section 4, we first give an upper bound for the essential dimension of smooth connected unipotent algebraic groups and then by combining with a result in Section 3, we prove Theorem \ref{thm:main}.

In the last section, we study smooth unipotent algebraic groups of essential dimension $0$.  Let $G$ be an smooth affine algebraic  group over a field $k$. It can be shown that $\ed_k(G)=0$ if and only if $G$ is {\it special}, 
%i.e., the first Galois cohomology $H^1(L,G)$ is trivial for every extenstion $L$ of $k$, 
i.e., for any field extension $L/k$, every $G$-torsor over $\Spec L$ is trivial, 
see \cite[Proposition 4.4]{Me} and \cite[Proposition 4.3]{TV}. Special groups are introduced by Serre in \cite{Se1}. Over algebraic closed fields, they are classified by Grothendieck \cite{Gro}.

Studying smooth unipotent algebraic groups of essential dimension $0$ is therefore equivalent to studying smooth unipotent algebraic groups which are special. It is well-known that over a perfect field $k$, every smooth connected unipotent group $G$ is $k$-split (see e.g. \cite[Chapter V, Corollary 15.5 (ii)]{Bo}), and hence special. Therefore, over a perfect field, a  smooth unipotent group is special if and only if it is $k$-split. (Note that a special algebraic group is always connected \cite{Se1}.)
%Now let $G$ be a unipotent algebraic $k$-group. Then  by Hilbert's Theorem 90, $\ed_k(G)=0$ (or equivalently, $G$ is special) if $G$ is $k$-split. We show that the converse also holds if $k$ is finitely generated over a perfect field. Namely, we have 
%Unipotent $k$-group which are $k$-split are  easy examples of special groups. It turns out that among unipotent $k$-groups, such unipotent $k$-groups are \emph{all} examples of special groups if the base field $k$ is finitely generated over a perfect field. 
It turns out that this statement still holds true over certain fields, e.g., fields which are finitely generated over a perfect field.
Namely, we have 
\begin{thm}
\label{thm:main2}
Let $k_0$ be a field of characteristic $p>0$, $v$ a  valuation of $k_0$. We assume that there is a $k_0^p$-basis $\{e_1,\ldots, e_n\}$ of $k_0$ such that $v(e_1),\ldots, v(e_n)$ are pairwise distinct modulo $p$. 
Let $k$ be a finite extension of $k_0$. Let $G$ be a non-trivial smooth unipotent algebraic $k$-group. Then $G$ is special if and only if $G$ is $k$-split.
\end{thm}
This theorem yields the following corollary (see Corollary \ref{cor:geometric} for a more general statement).
\begin{cor}
\label{cor:main3}
 Let $k$ be a field which is finitely generated  over a perfect field. Let $G$ be a non-trivial smooth unipotent algebraic $k$-group. Then $\ed_k(G)=0$ if and only if $G$ is $k$-split.
\end{cor}

To prove Theorem \ref{thm:main2} and Corollary \ref{cor:main3}, we need some results concerning the images of additive maps over valued fields. These results are presented in Section 5. 

We do not know whether Theorem \ref{thm:main2} is still true over an \emph{arbitrary} field $k$. 
\begin{question}
\label{question:ed0}
 Let $k$ be a field, $G$ a smooth unipotent algebraic $k$-group. Is this true that $\ed_k(G)=0$ if and only if $G$ is $k$-split? Equivalently, is this true that $G$ is special if and only if $k$-split?
\end{question}

\noindent {\bf Acknowledgements:} We would like to give our sincere thanks to H\'el\`ene Esnault for her support and constant encouragement. We would like to thank \NQT for his interest in the paper.

\section{A technical result}
For a smooth algebraic group over a field $k$, the flat cohomology $H^1_{\rm fppf}(K,G)$ is the same as the Galois cohomology $H^1(K,G)$ for any field extension $K/k$. We need the following lemma.
\begin{lem}
\label{lem:surj}
 Let $k$ ba a field. $G$ a smooth affine algebraic $k$-group. Let $U$ be a normal unipotent $k$-subgroup of $G$. Then the natural map 
$$\varphi: H^1(k,G)\to H^1(k,U)$$
is surjective.

Furthermore, if in addition that $U$ is $k$-split then $\varphi$ is a functorial bijection.
\end{lem}

\begin{proof}
See \cite[Chapter IV, 2.2, Remark 3]{Oe} for the first statement.

See \cite[Lemma 7.3]{GM} for the second statement.
\end{proof}

We have following key technical result, which is motivated by \cite[Lemma 3.4]{TV}.
 \begin{prop}
\label{prop:tech}
 Let $k$ be a field and consider an exact sequence of smooth affine algebraic $k$-groups
$$1\to A \to B \to C \to 1,$$
where $A$ is a unipotent normal subgroup of $B$. Let $K/k$ be a field extension and $x$ an element in $H^1(K,B)$. Then there exists a subfield extension $k\subset E \subset K$ and a twisted form $\tilde{A}$ of $A_E=A \times_{k}E$, $\tilde{A}$  is defined over $E$, such that 
\[\ed(x)\leq \ed_k(C)+ \ed_E(\tilde{A}).\]
Further, if $A$ is central in $B$ then one can choose $\tilde{A}=A_E$ and in particular
\[\ed_k(B)\leq \ed_k(C)+\ed_k(A).\]
\end{prop}

\begin{proof} 
Denote by $g_{-}:H^1(-,B)\to H^1(-,C)$ the natural morphism of functors induced from $B\to C$. Set $y=g_K(x)\in H^1(K,C)$, then there exists a subfield extension $k\subset E \subset K$  and $z$ in $H^1(E,C)$ such that ${\rm trdeg}(E:k) = \ed(y) \leq \ed_k(C)$ and the image of $z$ via $H^1(E,C)\to H^1(K,C)$ is equal to $y$. Since the natural map 
$g_E: H^1(E,B)\to H^1(E,C)$ 
is surjective, there exists $t$ in $H^1(E,B)$ such that $g_E(t)=z$. Let $b$ be a cocycle in $Z^1(E,B)$ representing $t$ and let $c$ be the image of $b$ in $Z^1(E,C)$. Denote by $_bA$, $_bB$ and $_cC$ the groups obtaining  by twisting $A$, $B$ and $C$ (more precisely, by twisting $A_E$, $B_E$ and $C_E$) using the cocycles $b$, $b$ and $c$ respectively. Then we get the following exact sequence of $E$-groups
\[1\to \,_bA \to \,_bB \to \,_cC \to 1\]
by twisting the initial sequence. 

Recall  that there is a functorial bijection between $H^1(L,\,_bH)$ and $H^1(L,\,_b H)$ for any $k$-group $H$, 1-cocycle $b:\Gal(k_s/k)\to H(k_s)$, and field extension $L/k$ (see \cite[I, 5.3, Proposition 35]{Se2}). Thus in the following commutative diagram, the maps $p, q, p^\prime, q^\prime$ are all bijective

{\tiny
\[
\xymatrix{
&&& H^1(E,B)  \ar@{>}^(.6){g_E}[rr]\ar@{>}'[d]_(0.6){\beta}[dd]
& & H^1(E,C) \ar@{>}_(.3){\gamma}[dd]
\\
&& H^1(E,\,_bB) \ar@{>}^p[ur] \ar@{>}^(0.6){g^\prime_E}[rr]\ar@{>}_(.3){\beta^\prime}[dd]
& & H^1(E,\,_cC) \ar@{>}^q[ur] \ar@{>}[dd]_(.3){\gamma^\prime}
\\
& {} \ar@{}'[r][rr] & & H^1(K,B) \ar@{>}'[r]^(1.2){g_K}[rr]
& & H^1(K,C)
\\
H^1(K,\,_bA) \ar@{}[ur] \ar@{>}^{f_K^\prime}[rr] && H^1(K,\,_bB) \ar@{>}^(0.6){g^\prime_K}[rr] \ar@{>}^{p^\prime}[ur]
& & H^1(K,\,_cC) \ar@{>}^{q^\prime}[ur]
}
\]
}
Note that the bottom row in the above diagram is an exact sequence of pointed sets. 

Since we twist by the cocycle representing $t$, we have $t=p(1)$,  where by abuse of notation, $1$ denote the trivial cohomology class. Since $p^\prime$ is bijective, there exists $x^\prime \in H^1(K,\, _bB)$ such that $x=p^\prime(x^\prime)$.
We have
\[
\begin{aligned}
y &=g_K(x)= g_K\circ p^\prime(x^\prime)=q^\prime \circ g^\prime_K (x^\prime) \\
&=\gamma(z)=\gamma(g_E(t))=\gamma\circ g_E \circ p(1)= q^\prime \circ g^\prime_K \circ \beta^\prime (1)=q^\prime (1).
\end{aligned}
\]
Since $q^\prime$ is bijective, $g^\prime_K(x^\prime)=1$. Hence there exists $u^\prime \in H^1(K,\,_bA)$ such that $x^\prime=f^\prime_K(u^\prime)$. By definition of $\ed_E(u^\prime)$, there is a subfield extension $E\subset E^\prime \subset K$ and an element $v^\prime\in H^1(E^\prime,\,_bA)$ such that ${\rm trdeg}(E^\prime:E)\leq \ed_E(\,_bA)$ and $u^\prime$ is the image of $v^\prime$ via $H^1(E^\prime,\,_bA)\to H^1(K,\,_bA)$. (Note that $\,_bA$ is only defined over $E$.)  From the following commutative diagram 

{\tiny

\[
\xymatrix{
& {}  
& & H^1(K,B) 
\\
H^1(K,\,_bA) \ar@{}[ur]\ar@{>}^{f^\prime_K}[rr]
& & H^1(K,\,_bB) \ar@{>}^{p^\prime}[ur]
\\
& {} 
& & H^1(E^\prime,B) \ar@{>}^{\beta_1}[uu]
\\
H^1(E^\prime,\,_bA) \ar@{>}^{f^\prime_{E^\prime}}[rr] \ar^{\alpha_1^\prime}[uu]
& & H^1(E^\prime,\,_cB) \ar@{>}^{\beta_1^\prime}[uu] \ar@{>}^{p^\prime_1}[ur]
},
\]
}
we get
\[
\begin{aligned}
x =p^\prime(x^\prime)=p^\prime \circ f^\prime_K (u^\prime)
= p^\prime \circ f^\prime_K \circ \alpha_1^\prime (v^\prime) =\beta_1\circ p^\prime_1\circ f^\prime_{E^\prime} (v^\prime).  
\end{aligned}
\]
Therefore, $x\in \im(\beta_1)$ and hence 
$$\ed(x)\leq {\rm trdeg}(E^\prime:k)= {\rm trdeg}(E^\prime:E)+ {\rm trdeg}(E:k)\leq \ed_k(C)+ \ed_{E}(\,_bA).$$

The second assertion follows immediately by construction since in the case that $A$ is central, by definition of twisting using a cocycle, we have $_bA=A$ as groups over $E$. 
\end{proof}

\begin{rmk} % The above proof is inspired by the notion of "fibration of functors" in \cite[Definition 1.12]{BF} and by the proof of \cite[Lemma 1]{TT1}.
 The twisted forms $\tilde{A}$ appreared in Proposition \ref{prop:tech} are also smooth unipotent algebraic groups.
\end{rmk}
%%%%%%%%%%%%%%%%%%%%%%%%%%%%%%%%%%%%%%%%%%%%%%%%%%%%%%%%%%%%%%%%%%%%%%%%%%%%%%%%%%%%%%%%%%%%%
%%%%%%%%%%%%%%%%%%%%%%%%%%%%%%%%%%%%%%%%%%%%%%%%%%%%%%%%%%%%%%%%%%%%%%%%%%%%%%%%%%%%%%%%%%%%%%
\section{Essential dimension of $p$-groups in characteristic $p$}
In this section, using Proposition \ref{prop:tech}, we derive some corollaries concerning the essential dimension of finite \'etale group schemes of order $p^n$ over a field of characteristic $p>0$, see Proposition \ref{prop:Ledet} and Proposition \ref{prop:JYL}. 

\subsection{Upper bound for finite \'etale unipotent groups}
The following result is obtained already by Ledet \cite{Le} in the case that $G$ is a finite abstract $p$-group, see also \cite[Theorem 1.4]{TV} for a more general result.

\begin{prop}
\label{prop:Ledet}
 Let $k$ be a field of characteristic $p$. Let $G$ be a finite \'etale $k$-group scheme of order $p^n$. Then $\ed_k(G)\leq n$.
\end{prop}
\begin{proof}
 We proceed by induction on $n$. If $n=1$ it is easy to see that $\ed_k(G)=\ed_k(\Z/p)=1$ (for example, see \cite[page 292]{BF}). Now since $G(k^s)$ is a $p$-group, $G$ has a central subgroup $H$ of order $p$. By Proposition \ref{prop:tech}, we get 
\[\ed_k(G)\leq \ed_k(H)+ \ed_k(G/H)=1+ \ed_k(G/H) \leq n,\]
since $\ed_k(G/H)\leq n-1$ by induction assumption.
\end{proof}

\begin{cor}
Let $k$ be a field of characteristic $p>0$. Let 
\[1\to P \to G \to A \to 1 \]
be an exact sequence of finite \'etale $k$-group schemes. Assume that $P$ is a finite \'etale $k$-group scheme of order $p^n$. Then 
\[\ed_k(A)\leq \ed_k(G)\leq \ed_k(A)+n.\]
\end{cor}

\begin{proof} The first inequality follows from Lemma \ref{lem:surj} and \cite[Lemma 1.9]{BF}.

For the second inequality, let $K/k$ be a field extension and $x$ an element in $H^1(K,G)$. By Proposition \ref{prop:tech}, there is a subfield extension $k\subset E\subset K$ and a twisted form $\tilde{P}$ of $P_E$ such that
\[ \ed(x)\leq \ed_k(A)+\ed_E(\tilde{P}).\]
By Proposition \ref{prop:Ledet}, $\ed_E(\tilde{P})\leq n$ (note that the orders of $\tilde{P}$, of $P_E$ and of $P$ are all equal). 
Therefore, $\ed(x) \leq \ed_k(A)+n$ and hence $\ed_k(G)\leq \ed_k(A)+n$. 
\end{proof}

\begin{rmk}
 Without the assumption of being $p$-group on $P$, it not true, in general, that $\ed_k(G)\geq \ed_k(G/P)$ (see \cite[Theorem 1.5]{MZ}).
\end{rmk}

\subsection{Elementary $p$-groups}
Let $k$ be a field of characteristic $p>0$. Let $G$ be a finite \'etale $k$-group scheme. It is called an {\it elementary $p$-group scheme} (over $k$) if it is of  $p$-power order, commutative and annihilated by $p$. 

\begin{lem}
\label{lem:elementary} 
Let $k$ be a field of characteristic $p>0$, $G$ an elementary finite \'etale $p$-group scheme over $k$. Then $\ed_k(G)$ is always less than or equal 2 and it is less than or equal 1 if $k$ is infinite.
\end{lem}
\begin{proof}
 If $k$ is infinite then by Lemma \ref{lem:1} (in the next section), $\ed_k(G)\leq 1$. 

Assume now that $k$ is finite. Let $K\supset k$ be any field extension of $k$ and $a$ an arbitrary element in $H^1(K,G)$. We show that $\ed(a)$ is always less than or equal 2. 

If $\ed(a)\leq 1$, then $\ed(a)<2$ trivially. 

If $\ed(a)\geq 1$ then there exist a field sub-extension $k\subset K_0 \subset K$ with $\trdeg_k(K_0)=\ed(a)$ and an element $x\in H^1(K_0,G)$ such that $x$ is sent to $a$ via $H^1(K_0,G)\to H^1(K,G)$. Since $\trdeg_k(K_0)=\ed(a)\geq 1$, $K_0$ contains $k(u)$, for some $u$, which is transcendental over $k$. Since $\ed_{k(u)}(G)\leq 1$, there is a subfield extension $k(u)\subset L \subset K_0$ with $\trdeg_{k(u)}(L)\leq 1$ and an element $y\in H^1(L,G)$ which is sent to $x$ via $H^1(L,G)\to H^1(K_0,G)$. Then $y$ is sent to $a$ via $H^1(L,G)\to H^1(K,G$), Therefore
\[\ed(a)\leq \trdeg_k(L)\leq 1+1=2.\]
So $\ed(a)$ is always less than or equal 2. Hence $\ed_k(G)\leq 2$.
\end{proof}

\subsection{Frattini subgroups}
Recall that the {\it Frattini subgroup} $\Phi(G)$ of a abstract finite group $G$ is the intersection of the maximal subgroups of $G$. It is a characteristic subgroup, i.e., it is invariant under every automorphism of $G$ and if $G\not= 1$ then $\Phi(G)\not=G$. If $G$ is $p$-group then $G/\Phi(G)$ is an elementary $p$-group.

To give a finite \'etale $k$-group scheme $G$ is the same as to give a finite abstract group $\G$ with  a continuous action of $\Gal(k_s/k)$ where $\Gal(k_s/k)$ acts as group automorphisms. Since the Frattini subgroup $\H=\Phi(\G)$ of $\G$ is invariant under the action of $\Gal(k_s/k)$, $\H$ with this Galois action defines a finite $k$-subgroup $H$ of $G$, it is also called the Frattini subgroup of $G$. If $G$ is a finite \'etale group scheme of order $p^n$, then $G/H$ is an (finite \'etale) elementary $p$-group scheme over $k$.

We obtain the following result, which is Theorem 8.4.1 in \cite{JLY} when $G$ is an abstract $p$-group.
\begin{prop}
\label{prop:JYL}
Let $k$ be a field of characteristic $p>0$,  $G$  a finite \'etale $k$-group scheme of order power of $p$ and let the order of its Frattini subgroup $\Phi(G)$ be $p^e$. 
\begin{enumerate}
\item If $k$ is infinite then $\ed_k(G)\leq e+1$.
\item If $k$ is finite then $\ed_k(G)\leq e+2$.
\end{enumerate}
\end{prop}
\begin{proof}
We have the following exact sequence of finite \'etale $k$-group schemes
\[1\to \Phi(G)\to G\to G/\Phi(G)\to 1,\]
with $N:=G/\Phi(G)$ is an elementary $p$-group. %By Lemma \ref{lem:elementary}, $\ed_k(N)$ is always less than or equal 2 and $\ed_k(N)\leq 1$ if $k$ is infinite. 

Let $K/k$ be a field extension and $x$ an element in $H^1(K,G)$. By Proposition \ref{prop:tech}, there is a subfield extension $k\subset E\subset K$ and a twisted form $\widetilde{\Phi(G)}$ of $\Phi(G)_E$ such that
\[ \ed(x)\leq \ed_k(N)+\ed_E(\widetilde{\Phi(G)}).\]
By Proposition \ref{prop:Ledet}, $\ed_E(\widetilde{\Phi(G)})\leq e$ (note that the order of $\widetilde{\Phi(G)}$ is equal to that of $\Phi(G)$). 
Therefore, $\ed(x) \leq e+ \ed_k(N)$ and hence $\ed_k(G)\leq e+\ed_k(N)$. The corollary now follows from Lemma \ref{lem:elementary}.
%The corollary now follows from Proposition \ref{prop:tech} and Proposition \ref{prop:Ledet}.
\end{proof}

\subsection{Homotopy invariance}
In \cite[Section 8]{BF} they prove the so-called {\it homotopy invariance} of essential dimension, that is $\ed_k(G)=\ed_{k(t)}(G)$, for algebraic groups defined over infinite fields.  In the next proposition, we show that this property does not hold for finite fields. Namely, we have

\begin{prop} Let $k=\mathbb{F}_p$ and $P$ an elementary $p$-group of rank $\geq 3$.
 Then \[\ed_{k(t)}(P)<\ed_k(P).\]
\end{prop}

\begin{proof}
 We consider $P$ as a constant group scheme over $k$. By Lemma \ref{lem:1},  $\ed_{k(t)}(P)\leq 1$.

On the other hand, $\ed_k(P)\geq 2$. In fact, assume for contradiction that $\ed_k(P)\leq 1$ then $P$ is isomorphic as an abstract group to a subgroup of ${\rm PGL}_2(\mathbb{F}_p)$ (see for example \cite[Lemma 7.2]{BF}). But this cannot happen since
\[{\rm Card}(G)\geq p^3> p(p^2-1)={\rm Card}({\rm PGL}_2(\mathbb{F}_p)).\]
Therefore, $\ed_k(P)>\ed_{k(t)}(P)$.
\end{proof}

\section{Upper bound for essential dimension of unipotent algebraic groups}
In this section we will prove Theorem \ref{thm:main}.

\subsection{Tits' structure theory of unipotent algebraic groups}
We first recall some results of Tits concerning the structure  of  unipotent algebraic groups over an arbitrary  (especially imperfect) field of positive characteristic, see \cite[Chapter V]{Oe} and \cite[Appendix B]{CGP}.
 
Let $G$ be a smooth unipotent algebraic group over a field $k$ of characteristic $p>0$. Then there exists a maximal central smooth connected $k$-subgroup of $G$ which is killed by $p$. This group is called {\it cckp-kernel} of $G$ and denoted by $cckp(G)$ or $\kappa(G)$. Here $\dim (\kappa(G))>0$ if $G$ is not finite. 

The following statements are equivalent:
\begin{enumerate}
\item $G$ is wound over $k$,
\item $\kappa(G)$ is wound over $k$.
\end{enumerate}
If the two equivalences are satisfied then $G/\kappa(G)$ is also wound over $k$ (\cite[Chapter V, 3.2]{Oe};  \cite[Appendix B, B.3]{CGP}).

\begin{prop}[{see \cite[B.3.3]{CGP}}] 
\label{prop:cckp} Let $k$ be a field of characteristic $p>0$. 
Let $G$ be a $k$-wound smooth connected unipotent algebraic $k$-group. Define the ascending chain of smooth connected normal $k$-subgroups $\{ G_i\}_{i\geq 0}$ as follows: $G_0=1$ and $G_{i+1}/G_i$ is the cckp-kernel of the $k$-wound group $G/G_i$ for all $i\geq 0$. These subgroups are stable under $k$-group automorphisms of $G$, their formation commutes with any separable extension of $k$, and $G_i=G$ for sufficiently large $i$.
\end{prop}

\begin{defn} The smallest natural number $i$ such that $G_i=G$ as in the previous proposition is called {\it the cckp-kernel length} of $G$ and denoted by $l=lcckp(G)$.

%The  series $G_0={1}\subset G_1\subset \cdots \subset G_{l-1}\subset G_l=G$ with $l=lcckp(G)$ as in the previous proposition is called the {\it ascending cckp-central series} of $G$.

Note that $lcckp(G)\leq \dim G$ since the cckp-kernel of a non-trivial smooth connected unipotent algebraic $k$-group is non-trivial.
\end{defn}

\begin{defn}
 Let $k$ be a field of characteristic $p>0$. A polynomial $P\in k[T_1,\ldots,T_r]$ is a $p$-{\it polynomial} if every monomial appearing in $P$ has the form $c_{ij} T_i^{p^j}$ for some $c_{ij}\in k$; that is $P=\sum_{i=1}^r P_i(T_i)$ with $P_i(T_i)=\sum_{j} c_{ij} T_i^{p^j}\in k[T_i]$. 

A $p$-polynomial $P\in k[T_1,\ldots,T_r]$ is called {\it separable} if it contains at least a non-zero monomial of degree 1.

If $P=\sum_{i=1}^r P_i(T_i)$ is a $p$-polynomial over $k$ in $r$ variables, then the {\it principal part} of $P$ is the sum of the leading terms of the $P_i$.
\end{defn}
\begin{prop}[{see \cite[Ch. V, 6.3, Proposition]{Oe} and \cite[Proposition B.1.13]{CGP}}]
\label{prop:Tits1}
 Let $k$ be a infinite field of characteristic $p>0$. Let $G$ be a smooth  unipotent algebraic $k$-group of dimension $n$. Assume that $G$ is commutative and annihilated by $p$. Then $G$ is isomorphic (as a $k$-group) to the zero scheme of a separable nonzero $p$-polynomial over $k$, whose principal part vanishes nowhere over $k^{n+1}\setminus\{0\}$.
\end{prop}

\subsection{Smooth connected unipotent algebraic groups}
In this section we give an upper bound for essential dimension of smooth connected algebraic groups, see Theorem \ref{thm:connected}. 
\begin{lem}
\label{lem:1}
 Let $k$ be an infinite field of characteristic $p>0$. Let $G$ be a smooth unipotent algebraic  $k$-group. Assume that $G$ is commutative and annihilated by $p$. Then $\ed_k(G)\leq 1$.
\end{lem}
\begin{proof} By a result of Tits (see Proposition \ref{prop:Tits1}), $G$ is isomorphic (as a $k$-group) to the zero scheme of a separable nonzero $p$-polynomial $f(T_1,\ldots,T_n)$, where $n=\dim G+1$, over $k$. That means we have the following exact sequence of $k$-groups
\[0\to G\to \G_a^n \stackrel{f}{\to}\G_a\to 0.  \]
 This follows that $H^1(K,G)=K/f(K)$ for any field extension $K/k$ and hence $\ed_k(G)\leq 1$. 
\end{proof}

\begin{thm} 
\label{thm:connected}
Let $G$ be a smooth connected algebraic unipotent group over a field $k$ of characteristic $p>0$, $G_s$ the $k$-split part of $G$. Let $l$ be the cckp-kernel length of  $G/G_s$. Then $\ed_k(G)\leq l$.
\end{thm}

\begin{proof} If $k$ is finite then $G$ is $k$-split and hence $\ed_k(G)=0\leq l$.

Now we assume that $k$ is infinite. 
By Lemma \ref{lem:surj}, the natural map $H^1(K,G)\to H^1(K,G/G_s)$ is a bijection for all field $K\supset k$. Therefore, $\ed_k(G)=\ed(G/G_s)$. Set $H=G/G_s$ and let $\{H_i\}_{i\geq 0}$ be the ascending chain of normal subgroups of $H$ as in Proposition \ref{prop:cckp} with $l=lcckp(H)$.

Since $H_{i+1}/H_i$ is the cckp-kernel of $H/H_i$, in particular, it is commutative and killed by $p$. Therefore, by Lemma \ref{lem:1}, $\ed_k(H_{i+1}/H_i)\leq 1$. Applying  Proposition \ref{prop:tech} to the following exact sequence
\[1\to H_{i+1}/H_i \to H/H_i \to H/H_{i+1}\to 1,\]
one has 
\[ \ed_k(H/H_i)\leq \ed_k(H/H_{i+1})+1,\]
for all $i=0,\ldots, l=lcckp(H)$. 
It implies that 
\[\ed_k(H)=\ed_k(H/H_0)\leq \ed_k(H/H_l)+l=l,\]
as required.
\end{proof}
The following result can be considered as a counterpart of Proposition \ref{prop:Ledet} for smooth connected unipotent algebraic groups.
\begin{cor} 
\label{cor:connected}
Let $G$ be a smooth connected unipotent algebraic group over a field $k$ of characteristic $p>0$.  Then $\ed_k(G)\leq \dim G$.
\end{cor}
\begin{proof} 
Let $G_s$ be the $k$-split part of $G$, $l$ the cckp-kernel of $G/G_s$. By Theorem \ref{thm:connected}, $\ed_k(G)\leq l$. The corollary then follows from the fact that cckp-kernel length $l$ of $G/G_s$ is less than or equal $\dim G/G_s\leq \dim G$.
\end{proof}

\begin{rmk} 
Corollary \ref{cor:connected} can also be proved by induction on $\dim G$ as follows: It is enough to consider the case  $k$ is infinite. Assume that this is the case. If $\dim G=1$, then $G$ is commutative and annihilated by $p$. Thus $\ed_k(G)\leq 1$ by Lemma \ref{lem:1}. Assume that $\dim G>1$. By \cite[Proposition 1]{TT2}, there exists a normal smooth connected $k$-subgroup $H$ of codimension 1 in $G$. Consider the following exact sequence
\[1\to H\to G\to G/H\to 1.\]
Let $K/k$ be a field extension and $x$ an element in $H^1(K,G)$. By Proposition \ref{prop:tech}, there is a subfield extension $k\subset E \subset K$ and a twisted form $\tilde{H}$ of $H_E$ such that 
\[ \ed(x) \leq \ed_k(G/H)+ \ed_E(\tilde{H}). \]
 By induction assumption, one has $\ed_E(\tilde{H})\leq \dim \tilde{H}=\dim H$. Therefore $\ed(x)\leq 1+\dim H=\dim G$ and hence $\ed_k(G)\leq \dim G$.
\end{rmk}

%\begin{cor}
% Let $k$ be a field of characteristic $p>0$. Let $G$ be a smooth connected $k$-wound algebraic $k$-group. Let $0=G_0\subset G_1\subset \cdots \subset G_{l-1}\subset G_l=G$, $l=lcckp(G)$, be the ascending cckp-kernel series of $G$. Then $\ed_k(G)\leq \dim (G_{l-1})+1$.
%\end{cor}

\begin{rmk}
 Fix a natural number $n$, Ledet conjectures that $\ed_k(\Z/p^n\Z) =n$ over any field $k$ of characteristic $p$. However, to the author's knowledge, there are no candidates for smooth connected unipotent algebraic groups and fields with the  essential dimension $n$. We would like to raise the following question.
\begin{question}
 For any natural number $n$, does there exist a field $k$ and a smooth connected unipotent $k$-group $G$ such that $\ed_k(G)=n$?
\end{question}
\end{rmk}

\subsection{Proof of Theorem \ref{thm:main}}
By Lemma \ref{lem:surj}, one has $\ed_k(G)=\ed_k(H)$. 
If $\Char k=0$ then it is well-known that $G$ is $k$-split, i.e., $H=G/G_s$ is trivial. Hence $\ed_k(G)=0$ and the theorem holds trivially.

We now assume that $k$ is of characteristic $p>0$. We consider  the following exact sequence of $k$-groups
\[1\to H^0\to H \to H/H^0 \to 1.\]
Let $K/k$ be a field extension and $x$ an element in $H^1(K,H)$. Then by Proposition \ref{prop:tech}, there is a subfield extension $k\subset E \subset K$ and a twisted form $\tilde{H^0}$ of $H^0_E$ such that 
\[\ed(x) \leq \ed_k(H/H^0)+ \ed_E(\tilde{H^0}).\]
By Corollary \ref{cor:connected}, 
\[\ed_k(\tilde{H^0}) \leq \dim \tilde{H^0} = \dim H^0_E =\dim G/G_s.\]
Hence, we have  the first inequality 
\[ \ed_k(G)\leq \ed_k(H/H^0)+\dim (G/G_s). \]
The second inequality follows immediately from Proposition \ref{prop:Ledet}.
%\end{proof}
\qed
%%%%%%%%%%%%%%%%%%%%%%%%%%%%%%%%%%%%%%%%%%%%%%%%%%%%%%%%%%%%%%%%%%%%%%%%%%%%%%%%%%%%%%%%
\section{Images of additive polynomials over valued fields}
In this section, we prove a result concerning the image of an additive polynomial over certain valued field, see Proposition \ref{prop:TT}, which is needed in proving Theorem \ref{thm:main2} in Section 6. %We first start with some lemmas.
\subsection{Some lemmas}
\begin{lem} 
\label{lem:val1}
Let $\Gamma$ be a nontrivial totally ordered commutative group 
\begin{enumerate}
\item For any element $\gamma$ in $\Gamma$, there exists $\beta\in \Gamma$ such that $\beta< \gamma$.
\item Let $\gamma_1,\ldots,\gamma_r$ be elements in $\Gamma$ and let $n_1,\ldots,n_r$ be positive numbers. Then there exists an element $\gamma_0$ in $\Gamma$  such that for all elements $\gamma<\gamma_0$, $\gamma\in \Gamma$, we have $n_i\gamma<\gamma_i$ for all $i$.
\end{enumerate}
\end{lem}
\begin{proof}
 %This is \cite[Lemma 2.1]{BT} and the proof is easy.
1) If $\gamma \geq 0$, then let $\beta <0 \leq \gamma$ (such an element exists since $\Gamma$ is nontrivial). 

If $\gamma<0$, one can takes $\beta=2\gamma<\gamma$.  \\
2) We set 
\[ \gamma_0:=\min\{\gamma_1,\ldots,\gamma_r, 0\}.\]
Now let $\gamma$ be an arbitrary element such that $\gamma<\gamma_0$.  Since $\gamma<\gamma_i$, $\gamma<0$, it follows that 
$n_i\gamma<\gamma_i$, for all $i$.
\end{proof}

\begin{lem}
\label{lem:val3}
Let $\Gamma$ be a totally ordered commutative group, $p$ a prime number, $d$ a natural number. Let $\alpha_0,\gamma_0$ be elements in $\Gamma$. Then there exist infinitely many elements $\gamma_i\in \Gamma$ such that 
\[ \gamma_0>\gamma_1>\cdots >\gamma_i>\cdots \]
and $\gamma_i \equiv \alpha_0$ modulo $p^d$ for all $i>0$.
\end{lem}
\begin{proof}
 By Lemma \ref{lem:val1}, there is $\gamma\in \Gamma$ such that $p^d\gamma<\gamma_0-\alpha$. We set $\gamma_1:=\alpha+p^d\gamma$. Then $\gamma_1<\gamma_0$ and $\gamma_1\equiv \alpha_0$ modulo $p^d$. Continuing this way, one can construct a sequence $\gamma_0>\gamma_1>\gamma_2>\cdots$ satisfies the requirement of the lemma.
\end{proof}

The following lemma is a generalization of \cite[Lemma 4.4.1]{TT1} from discrete valuation to arbitrary valuation. Using some modifications, the proof in \cite{TT1} works well in our case.  Because the proof is quite technical, we would like to give it here in detail for reader's convinence.
%to convince the reader (and ourself).
\begin{lem}
\label{lem:TT}
Let $k$ be a field of characteristic $p>0$, $v$ a non-trivial valuation of $k$ with the value group $\Gamma$. Let $P=\sum_{i=1}^r \sum_j c_{ij}T_i^{p^j}$ a non-trivial $p$-polynomials in $r$ variables with coefficients in $k$. Let $P_{princ}=\sum_{i=1}^r c_i T_i^{p^{m_i}}$ be the principal part of $P$. Assume that for all $(a_1,\ldots, a_r)\in k\times \cdots \times k$ ($r$ times), 
$v(c_i)+p^{m_i}v(a_i)$ are all distinct whenever they are defined. Then there exists a constant $C_0$ depending only on $P$ such that if $a=P(a_1,\ldots, a_r)$ and $v(a)< C_0$  then $v(a)=v(c_i)+p^{m_i}v(a_i)$, for some $i$.
\end{lem}

\begin{proof} 
We process by induction on $r$. First let $r=1$, $P(T)=b_0T+\cdots +b_mT^{p^m}$, $b_m\not=0$. Set $I:=\{i\mid b_i\not=0\}\subset \{0,\ldots,n\}$. By Lemma \ref{lem:val1}, there exists $A \in \Gamma$ such that
\[
 (p^m-p^i)A < v(b_i)-v(b_m), \forall \, i\in I\setminus\{m\}.
\]
We set $$B=\min_{i\in I}\{Ap^i+v(b_i)\},$$
and pick any $C_0$ with $C_0<B$. Now assume that $a=P(a_1)$ ($a_1\in k$) such that $v(a)\leq C_0$. Let $i_0$ be such that 
\[
v(b_{i_0}a_1^{p^{i_0}})=\min_{i\in I} \{v(b_ia_1^{p^i})\}.
\]
Then we have $C_0\geq v(a)=v(P(a_1))\geq v(b_{i_0}a_1^{p^{i_0}})=v(b_{i_0})+p^{i_0}v(a_1).$ Hence by the choices of $C_0$ and  of $B$, one has
\[
p^{i_0}v(a_1)\leq C_0-v(b_{i_0})<B-v(b_{i_0})\leq Ap^{i_0}.
\]
This implies that $v(a_1)<A$ and by the definition of $A$, 
\[
(p^m-p^i)v(a_1)< v(b_i)-v(b_m),\; \forall \, i\in I\setminus\{m\},
\]
or equivalently,
\[
v(b_ia_1^{p^i})=v(b_i)+p^iv(a_1)>v(b_m)+p^mv(a_1)=v(b_ma_1^{p^m}),\;\forall\, i\in I\setminus\{m\}.
\]
Therefore $v(a)=v(b_m)+p^mv(a_1)$ as required. 

Now assume that $r>1$ and that the assertion of the lemma holds true for all integers less than $r$. By induction hypothesis, for any $l$ with $1\leq l <r$, there exist constants $B_l$ (in the value group $\Gamma$) satisfying the lemma for the case $r=l$. Any monomial of $P(T_1,\ldots,T_r)-P_{princ}(T_1,\ldots, T_r)$ is of the form $\lambda T_j^{p^{m_j-s}}$ with $\lambda \in k^\times$, $1\leq j\leq r$, $1\leq s$, and for such a monomial we choose an element $a_{\lambda, s,j}$ in $\Gamma$ such that
\[
(p^{m_j}-p^{m_j-s})a_{\lambda, s.j}<v(\lambda)-v(c_j).
\]
(The existence of such an element is ensured by Lemma \ref{lem:val1}.) Also by Lemma \ref{lem:val1}, we can choose $C_3$ and $C_2$ in $\Gamma$ such that
\[
\begin{aligned}
p^{m_i}C_3 & < v(\lambda)+p^{m_j-s}a_{\lambda,j,s}-v(c_i), \;\forall \, \lambda, j, s;\\
p^{m_j}C_2 & < v(c_i)+ p^{m_i}C_3-v(c_j), \;\forall \, i,j.
\end{aligned}
\]
Let 
\[
\begin{aligned}
 C_1&=\min_{i,j}\{ v(c_{ij})+p^j C_2\},\\
C_0&=\min\{C_1,B_1,\ldots,B_{r-1}\}.
\end{aligned}
\]
Assume that $a=P(a_1,\ldots,a_r)$, $a_i\in k$ and $v(a)< C_0$. If there exists $i$ such that $a_i=0$ then the cardinality of the set $\{i\mid a_i\not=0\}$ is less than $r$ and instead of $P$ we can consider the polynomial
\[
\tilde{P}=P(T_1,\ldots,T_{i-1},0,T_{i+1},\ldots,T_r)
\]
in $r-1$ variables and use the induction hypothesis. So we assume that $a_i\not=0$ for all $i$. Let
\[
i_0=\min_{1\leq i\leq r}\{i \mid v(a_i)\leq v(a_j), \text{ for all $j$}, 1\leq j\leq r\}.
\]
Then 
\[
v(a)=v(P(a_1,\ldots,a_r))\geq \min\{ v(c_{ij}a_i^{p^j})\}\geq \min\{v(c_{ij})+p^j v(a_{i_0})\}.
\]
By assumption $v(a)< C_0\leq C_1$, that implies that, for some $i,j$, one has
\[
v(c_{ij})+p^j v(a_{i_0})<C_1\leq v(c_{ij})+p^j C_2.
\]
Hence $v(a_{i_0})<C_2$. Since $v(c_i)+p^{m_i}v(a_i)$ are pairwise distinct, there exists a unique $i_1$ such that
\[
v(c_{i_1})+p^{m_{i_1}}v(a_{i_1})=\min_{1\leq j\leq r} \{v(c_j)+p^{m_j}v(a_j)\}.
\]
Since
\[ 
v(c_{i_1})+p^{m_{i_1}}v(a_{i_1})\leq v(c_{i_0})+p^{m_{i_0}}v(a_{i_0})< v(c_{i_0})+p^{m_{i_0}}C_2,
\]
one has $v(a_{i_1})<C_3$, since otherwise we would have
\[
v(c_{i_1})+p^{m_{i_1}}v(a_{i_1})\geq v(c_{i_1})+p^{m_{i_1}}C_3\geq v(c_{i_0})+p^{m_{i_0}}C_2
\]
which contradicts the above inequalities. 

Now we show that
\[
v(P(a_1,\ldots,a_r))=v(c_{i_1})+p^{m_{i_1}}v(a_{i_1}).
\]
This follows from  two facts below:

(i) For any monomial $\lambda T_j^{p^{m_j-s}}$ of $P(T_1,\ldots,T_r)-P_{princ}(T_1,\ldots,T_r)$, $\lambda\in k^\times$, $1\leq j\leq r$, $1\leq s$, if $v(a_j)<a_{\lambda,j,s}$ then by the definitions of $a_{\lambda,j,s}$ and of $i_1$ one has
\[ v(\lambda a_j^{p^{m_j-s}})=v(\lambda)+p^{m_j-s}v(a_j)>v(c_j)+p^{m_j}v(a_j)\geq v(c_{i_1})+ p^{m_1}v(a_{i_1}).\]
Also, if $v(a_j)\geq a_{\lambda,j,s}$ then again by definitions of $a_{s,j,s}$ and of $C_s$ one has
\[
v(\lambda a_j^{p^{m_j-s}})\geq v(\lambda)+p^{m_j-s}a_{\lambda,j,s}>v(c_{i_1})+p^{m_{i_1}}C_3> v(c_{i_1})+p^{m_{i_1}}v(a_{i_1}),
\]
since $v(a_{i_1})<C_3$. 

Thus  one always has
\[
v(\lambda a_j^{p^{m_j-s}})> v(c_{i_1})+p^{m_1}v(a_{i_1}).
\]

(ii) For $j\not=i_1$, by the uniqueness of $i_1$ one has
\[
v(c_ja_j^{p^{m_j}})>v(c_{i_{1_1}})+p^{m_{i_1}}v(a_{i_1}).
\]
Hence
\[
v(P_{princ}(a_1,\ldots,a_r))=v\left(c_{i_1}a_{i_1}^{p^{m_{i_1}}}+\sum_{j\not=i_1}c_ja_j^{p^{m_j}}\right)=v(c_{i_1}a_{i_1}^{p^{m_{i_1}}}).
\]
Now (i) and (ii) imply that
\[
 \begin{aligned}
  v(a)&=v(P(a_1,\ldots,a_r)\\
&=v\left( P_{princ}(a_1,\ldots,a_r)+(P(a_1,\ldots,a_r)-P_{princ}(a_1,\ldots,a_r))\right)\\
%&=v\left(\left( c_{i_1}a_{i_1}^{p^{m_{i_1}}}+\sum_{j\not=i_1}c_ja_j^{p^{m_j}}\right)+P(a_1,\ldots,a_r)-P_{princ}(a_1,\ldots,a_r))\right)\\
&=v(c_{i_1}a_{i_1}^{p^{m_{i_1}}})=v(c_{i_1})+p^{m_{i_1}}v(a_{i_1}).
 \end{aligned}
\]
The proof of the lemma is completed.
\end{proof}

\subsection{Valuation basis}
\begin{defn}
\label{defn:valuation basis}
 Let $(k,v)$ be a valued field of characteristic $p>0$, $d$ a natural number. A system $(b_i)_{i\in I}$ of non-zero elements in $k$ is called $k^{p^d}$-{\it valuation independent} 
with respect to (w.r.t)  the valuation $v$   if the values $v(b_i), i\in I$ are all pairwise distinct modulo $p^d$. 

If $V$ a $k^{p^d}$-vector subspace of $k$, this system is called {\it valuation basis} of $V$ if it generates $V$ as $k^{p^d}$-vector space and it is $k^{p^d}$-valuation independent.
\end{defn}
\begin{rmks} (1) Notations being as above. If $(b_i)_{i\in I}$ is $k^{p^d}$-valuation independent then  it is $k^{p^d}$-linearly independent (see the proof of Lemma \ref{lem:independence} (2) below). In particular, a valuation basis of $V$ is a basis of $V$ as $k^{p^d}$-vector space.

 (2) Our definitions of valuation independence and  of valuation basis are slightly different from those in \cite{DK}. A valuation basis in our sense is a valuation basis in their sense.
\end{rmks}

\begin{lem} 
\label{lem:independence}
Let $k$ be a field of characteristic $p>0$, $v$ a non-trivial valuation of $k$. Let $n,d$ be natural numbers. 
\begin{enumerate}
 \item Suppose that there are $n$ elements of $k$ which are $k^p$-valuation independent with respect to $v$.  Then  there are $n^d$ elements which are $k^{p^d}$-valuation independent with respect to $v$.
\item If $k$ has a finite $k^p$-valuation basis with respect to $v$ then $k$ has a finite $k^{p^d}$-valuation basis with respect to $v$.
\end{enumerate}

\end{lem}

\begin{proof} (1) We process by induction on $d$. By assumption, the statement (1) is true for $d=1$. 

Now we assume that $d\geq 2$ and that the assertion of (1) is true for $d-1$, i.e., there is a $k^{p^{d-1}}$-basis $g_1,\ldots,g_{n^{d-1}}$  such that $v(g_1),\ldots,v(g_{n^{d-1}})$ are pairwise distinct modulo $p^{d-1}$. 

Let $e_1,\ldots, e_n$ be elements of $k$ such that $v(e_1),\ldots, v(e_n)$ are pairwise distinct modulo $p$. 

For each pair $i,j$ with $1\leq i\leq n^{d-1}, 1\leq j\leq n$ we define $u_{ij}=g_ie_j^{p^{d-1}}$. Then there are $n^d$ such of $u_{ij}$ and these $v(u_{ij})$ are pairwise distinct modulo $p^d$. In fact, if $v(u_{ij})\equiv v(u_{i^\prime j^\prime})$ modulo $p^d$ for two pairs $(i,j)$ and $(i^\prime,j^\prime)$ then 
\[
v(g_i)-v(g_{i^\prime})+ p^{d-1}(v(e_j)-v(e_{j^\prime}))\equiv 0 \mod p^d.
\]
In particular $v(g_i)- v(g_{i^\prime})\equiv 0 \mod p$, hence $i=i^\prime$. This implies that $v(e_j)-v(e_{j^\prime})\equiv 0\mod p$ and $j=j^\prime$. 

(2) We first note that such $u_{ij}$ are $k^{p^d}$-linear independent. In fact, assume that there is a non-trivial $k^{p^d}$-linear combination $\sum a^{p^d}_{ij} u_{ij}=0$. Since all value $v(a^{p^d}_{ij}u_{ij})=p^{d}v(a_{ij})+v(u_{ij})$ are pairwise distinct whenever they are defined, one has 
\[v(0)=v(\sum a^{p^d}_{ij} u_{ij})=p^dv(a_{i_0j_0})+v(u_{i_0j_0}),\] 
for some pair $(i_0,j_0)$, it is impossible. 

Now (2)  follows from the part (1) and the fact that $[k:k^{p^d}]=[k:k^p]^d$ (by induction on $d$).
\end{proof}

\begin{lem}
\label{lem:basis}
Let $k$ be a field of characteristic $p>0$, $v$ a non-trivial valuation of $k$ and let $d$ be a natural number. 
We assume that $k$ has a finite $k^{p^d}$-valuation basis with respect to $v$. Let $V$ be a $k^{p^d}$-vector subspace of $k$. Then $V$ has a (finite) $k^{p^d}$-valuation basis with respect to $v$.
\end{lem}
\begin{proof}
Let $N=[k:k^{p^d}]$ and $u_1,\ldots,u_{N}$ a $k^{p^d}$-valuation basis of $k$. Let $b_1,\ldots,b_s$ be a $k^{p^d}$-basis of $V$ Then
for each $i$, we can write
\[
b_i=a_{i1}^{p^d}u_1+\cdots+a_{iN}^{p^d}u_N,
\]
where $a_{ij}$ are elements in $k$. Since $v(u_j)$ are pairwise distinct modulo $p^d$, the values  $v(a_{ij}^{p^d}u_j)$ are pairwise distinct. Hence there is a unique index $j_1$ such that $v(b_1)=v(a_{ij_1}^{p^d}u_{j_1})$. In particular $a_{1j_1}\not=0$.

We set $b_1^\prime:=b_1$ and for each $i\geq 2$, we set 
$b_i^\prime:=b_i- (a_{ij_1}/a_{1j_1})^{p^d} b_1$. Then $b_1^\prime,\ldots, b_s^\prime$ form a $k^{p^d}$-basis of $V$. Moreover, terms of the form $\lambda^{p^d} u_{j_1}$ do not appear in $b_2^\prime,\ldots, b_s^\prime$. Similarly, for each $i\geq 2$, we can write

\[
b^\prime_i=(a^\prime_{i1})^{p^d}u_1+\cdots+(a^\prime_{iN})^{p^d}u_N,
\]
where $a^\prime_{ij}$ are elements in $k$. And there is a unique index $j_2$ such that $v(b^\prime_2)=v((a^\prime_{ij_1})^{p^d}u_{j_2})$. We set $b_2^{\prime\prime}:=b_2^\prime$ and
$b_i^{\prime\prime}:=b^\prime_i- (a^\prime_{ij_2}/a^\prime_{2j_2})^{p^d} b^\prime_2$.
 Note that $j_2\not=j_1$ and terms of forms $\lambda_1^{p^d} u_{j_1}$ and of forms $\lambda_2^{p^d} u_{j_2}$ do not appear in $b_3^{\prime\prime},\ldots, b_s^{\prime\prime}$.

Continuing this way by modifying $b_3^{\prime\prime},\ldots,b_s^{\prime\prime}$ and so on, we obtain a $k^{p^d}$-basis $c_1,\ldots,c_s$ of $V$ such that $v(c_1),\ldots,v(c_s)$ are pairwise distinct modulo $p^d$.
\end{proof}

\subsection{A lemma of Dries and Kuhlmann}
The following lemma is a generalization of \cite[Lemma 4]{DK}. They treat the case of local fields, i.e., complete discrete valued fields with finite residue field. With the help of Lemma \ref{lem:basis}, their proof  can be extended to our case.  We include it here for the reader's convenience.
\begin{lem}
 \label{lem:DK}
 
Let $k$ be a field of characteristic $p>0$, $v$ a non-trivial valuation of $k$. We assume that $k$ has a finite $k^p$-valuation basis.  Let $P=f_1(T_1)+\cdots+f_r(T_r)$ be an additive (i.e. $p$-) polynomial with  coefficients in $k$ in $r$ variables, the principal part of which vanishes nowhere over $k^r\setminus\{0\}$. Let $S=\im(P)=f_1(k)+\cdots+f_r(k)$. Let $p^{d_i}=\deg f_i$, $p^d=\max p^{d_i}$, and $s=\sum_{i=1}^r n^{d-d_i}$ where $n:=[k:k^p]$. Then there are $s$ additive polynomials $g_1,\ldots,g_s\in k[X]$ in one variable $X$ such that
\begin{enumerate}
\item  $S=g_1(k)+\cdots+g_s(k)$;
\item all polynomials $g_i$ have the same degree $p^d$;
\item the leading coefficients $b_1,\ldots, b_s$ of $g_1,\ldots, g_s$ are such that $v(b_1),\ldots, v(b_s)$ are distinct modulo $p^d$.
\end{enumerate}
\end{lem}

\begin{proof} 
By Lemma \ref{lem:independence}, for each $i$, there are $n^{d-d_i}$ elements $u_{i1},\ldots, u_{i,n^{d-d_i}}$ such that these elements form a $k^{p^{d-d_i}}$-basis of $k$ and $v(u_{i1}),\ldots,v(u_{i,n^{d-d_i}})$ are pairwise distinct modulo $p^{d-d_i}$. In particular, we can write
\[
k=u_{i1}k^{p^{d-d_i}}+\cdots+u_{i,n^{d-d_i}}k^{p^{d-d_i}}.
\]
Hence
\[
f_i(k)=f_i(u_{i1}k^{p^{d-d_i}})+\cdots+f_i(u_{i,n^{d-d_i}}k^{p^{d-d_i}})=h_{i1}(k)+\cdots+h_{i,n^{d-d_i}}(k)
\]
where
\[ h_{ij}(X):=f_i(u_{ij}X^{p^{d-d_i}})\in k[X].\]
And  then
\[ S=\sum_{i=1}^r \sum_{j=1}^{n^{d-d_i}}h_{ij}(k) \]
with all polynomials $h_{ij}$ having degree $p^d$.

We claim that the leading coefficients $c_{ij}=c_i u_{ij}^{p^{d_i}}$ of the polynomials $h_{ij}$ are $k^{p^d}$-linearly independent. In fact, assume that for $a_{ij}\in k$,
\[
0=\sum_{i=1}^r \sum_{j=1}^{n^{d-d_i}} c_{ij}a_{ij}^{p^d}=\sum_{i=1}^r c_i \sum_{j=1}^{n^{d-d_i}}u_{ij}^{p^{d_i}}a_{ij}^{p^d}= 
\sum_{i=1}^r c_i \left(\sum_{j=1}^{n^{d-d_i}}u_{ij}a_{ij}^{p^{d-d_i}}\right)^{p^{d_i}}.
\]
By assumption that the principal part of $P$ vanishes nowhere over $k^r\setminus \{0\}$, one has
\[ 
\sum_{j=1}^{n^{d-d_i}}u_{ij}a_{ij}^{p^{d-d_i}}=0 \; \text{ for } 1\leq i\leq r.
\]
Since $u_{i1},\ldots,u_{i,n^{d-d_i}}$ are $k^{p^{d-d_i}}$-linearly independent, $a_{ij}=0$ for all $i$ and $j$.

We have now found $s=\sum_{i=1}^rn^{d-d_i}$ additive (i.e., $p$-) polynomials $\tilde{h}_1,\ldots, \tilde{h}_s$ in $k[X]$ with $k^{p^d}$-linearly independent leading coefficients $\tilde{c}_1,\ldots,\tilde{c}_s$ and such that $S=\tilde{h_1}(k)+\cdots +\tilde{h}_s(k)$. The Lemma \ref{lem:basis} shows that the $k^{p^d}$-vector space generated by $\tilde{c}_1,\ldots, \tilde{c}_s$ admits a $k^{p^d}$-basis $b_1,\ldots,b_s$, say, for which $v(b_1),\ldots,v(b_s)$ are pairwise distinct modulo $p^d$. Write $b_i=\sum_{j=1}^s r_{ij}^{p^d} \tilde{c}_j$ and we set
\[
g_i(X):= \sum_{j=1}^s \tilde{h}_j(r_{ij}X)
\]
and observe that for each $i$ the polynomial $g_i$ is of degree $p^d$ with leading coefficient $b_i$. 

It only remains to show that the condition (1) is satisfied. Since $S$ is an additive subgroup of $K$ and contains the images $\tilde{h}_j(k)$ for all $j$ it follows that 
\[
g_1(k)+\cdots+g_s(k)\subset \tilde{h}_1+\cdots+\tilde{h}_s(k)=S.
\]
On the other hand, both $\tilde{c}_1,\ldots,\tilde{c}_s$ and $b_1,\ldots,b_s$ are bases, so the matrix $(r_{ij}^{p^d})$ is invertible. Hence, the matrix $(r_{ij})$ is also invertible. Denote its inverse by $(s_{ij})$, with $s_{ij}\in k$. One can check that 
\[
\tilde{h}_i=\sum_{j=1}^s g_j(s_{ij}X).
\]
Hence $S= \tilde{h}_1+\cdots+\tilde{h}_s(k) \subset g_1(k)+\cdots+g_s(k)$, which concludes the proof.

\end{proof}

\subsection{Images of $p$-polynomials} 
\begin{lem}
\label{lem:equality}
 Let $k$ be a field of characteristic $p>0$, $v$ a non-trivial valuation of $k$ with value group $\Gamma$. Assume that $k$ has a finite $k^p$-valuation basis then we have 
\[[k:k^p]=[\Gamma:p\Gamma]=p^m,\]
for some natural number $m$.
\end{lem}

\begin{proof}
Let $N=[k:k^p]$. Consider a finite set of elements $\gamma_1,\ldots,\gamma_{N^\prime}$, $\gamma_i\in \Gamma$,  which are representatives of cosets of $p\Gamma$ in $\Gamma$. 
For each $i$, choose an element $b_i\in k$ such that $v(b_i)=\gamma_i$. As $v(b_1),\ldots,v(b_{N^\prime})$ are pairwise distinct modulo $p$, it implies that $b_1,\ldots,b_{N^\prime}$ are $k^p$-linearly independent. In particular $N^\prime\leq N$. Hence $[\Gamma:p\Gamma]$ is finite and  $M:=[\Gamma:p\Gamma]\leq N$.

On the other hand, let $e_1,\ldots, e_N$ a $k^p$-valuation basis of $p$. Since $v(e_1),\ldots, v(e_N)$ are pairwise distinct modulo $p$, we have $N\leq M$.  Therefore $N=M$

Finally, note that $k/k^p$ is a finite $\F_p$-vector space, so $N=p^m$, for some $m$.
\end{proof}

Now we have the following result, which plays an important role in the proof of Theorem \ref{thm:main2} in the last section.

\begin{prop}
 \label{prop:TT}
 Let $k$ be a field of characteristic $p>0$, $v$ a non-trivial valuation of $k$ with value group $\Gamma$. We assume that $k$ has a finite $k^p$-valuation basis and set $p^m:=[k:k^p]$. 
 Let $P$ be a $p$-polynomial  in $r$ variables with coefficients in $k$ satisfying the condition that the principal part $P_{princ}=\sum_{i=1}^r c_i T_i^{p^{m_i}}$, $c_i\in k^*$, vanishes nowhere over $k^r\setminus\{0\}$. Let $d=\max m_i$. Then we have 
\[s:=\sum_{i=1}^r p^{m(d-m_i)}\leq p^{md}.\]
Furthermore, if $s< p^{md}$ then the quotient $k/P(k)$ is infinite.
\end{prop}
\begin{proof}
We write 
\[P(T_1,\ldots,T_r)=f_1(T_1)+\cdots+f_r(T_r),\]
where each $f_i$ is a $p$-polynomial in one variable $T_i$ with coefficients in $k$ and of degree $p^{m_i}$. We set
\[S=\im(P)=f_1(k)+\cdots+f_r(k).\]
 
Choose $g_1,\ldots, g_s$ with leading coefficients $b_1,\ldots,b_s$, for which $v(b_1),\ldots,v(b_s)$ are pairwise distinct modulo $p^d$ as in Lemma \ref{lem:DK}. We set
\[Q(T_1,\ldots,T_s)=g_1(k)+\cdots+g_s(k).\]
Then $S=\im(Q)$.

By Lemma \ref{lem:equality}, $\Gamma/p\Gamma$ is of order $p^m$ and $\Gamma/p^d\Gamma$ is of  order $p^{md}$ by induction on $d$. As $v(b_1),\ldots,v(b_s)$ are pairwise distinct modulo $p^d$, it implies in particular that $s\leq p^{md}$.

Now we assume that $s<p^{md}$. Then there is an element $l\in \Gamma$ such that $v(b_i)\not\equiv l$ mod $p^d$ for all $i\in \{1,\ldots, s\}$. 

Since $v(b_1),\ldots,v(b_s)$ are pairwise distinct modulo $p^d$, for any tuple $(a_1,\ldots,a_s)\in k^\times \times\cdots\times k^\times$ ($s$ times), the values $v(b_i)+p^dv(a_i)$, $1\leq i\leq s$, are pairwise distinct. Then all conditions in Lemma \ref{lem:TT} are satisfied (for the $p$-polynomial $Q$), so there is $C_0$ as in the lemma. 

We claim that for all $a\in k$ with $v(a)\leq C_0$ and $v(a)\equiv l$ modulo $p^d$, $a$ is not in $S=\im Q$. In fact, assume that $a=Q(a_1,\ldots,a_s)$.  By Lemma \ref{lem:TT}, there is an index $i$ such that $v(a)=v(b_i)+p^d v(a_i)$. But this contradicts to the fact that $v(b_i)\not\equiv l$ modulo $p^d$, hence the claim follows.

By Lemma \ref{lem:val3}, we can choose a sequence $(e_i)_i$, $e_i\in k$ for all $i\geq 1$ such that
\[C_0> v(e_1)>v(e_2)>\cdots > v(e_i)>\cdots \]
and $v(e_i)\equiv l$ modulo $p^d$ for all $i$. 
Then $v(e_i-e_{i+j})=v(e_{i+j})\equiv l$ modulo $p^d$ for all $i,j\geq 1$. By the claim above, $e_i-e_{i+j}\not\in \im(Q)=S$ for all $i,j\geq 1$. 
Hence all $e_i$ have distinct images in $k/\im(Q)=k/\im(P)$. Therefore $k/\im(P)$ is infinite as required.
\end{proof}

\section{Unipotent groups of essential dimension $0$}
In this section, we will prove Theorem \ref{thm:main2} and Corollary \ref{cor:main3} stated in the Introduction.

\subsection{Infiniteness of Galois cohomology of unipotent algebraic groups over valued fields}
\begin{prop}
\label{prop:dim<p-1}
Let $k$ be a field of characteristic $p>0$, $v$ a non-trivial valuation of $k$. We assume that k has a finite $k^p$-valuation basis (see Definition \ref{defn:valuation basis}).
 Let $G$ a non-trivial smooth connected unipotent algebraic $k$-group of dimension $<[k:k^p]-1$. If $G$ is not split over $k$ then $H^1(k,G)$ is infinite.
\end{prop}

\begin{proof}
Let $G_s$ be the $k$-split part of $G$. Then $G/G_s$ is nontrivial, connected, $k$-wound and $H^1(k,G)=H^1(k,G/G_s)$ by Lemma \ref{lem:surj}. 
So we may assume that $G$ is wound over $k$. By \cite[Chapter V, 3.3]{Oe}, $G$ has a composition series of characteristic subgroups defined and wound over $k$: $G=G_0\supset G_1\supset \cdots \supset G_u={1}$ such that each quotient $G_i/G_{i+1}$ is commutative, $k$-wound and annihilated by $p$. 
Also by Lemma \ref{lem:surj}, we may assume further that $G$ is commutative, wound over $k$ and annihilated by $p$. In this case, $G$ is $k$-isomorphic to a $k$-subgroup of $\G_a^r$, where $r=\dim G +1$, which is the zero set of a separable $p$-polynomial $P(T_1,\ldots,T_r)\in k[T_1,\ldots,T_r]$, whose the principal part $P_{princ}=\sum_{i=1}^r c_i T_i^{p^{m_i}}$ vanishes nowhere over $k^r\setminus \{0\}$, see Proposition \ref{prop:Tits1}. 

By Proposition \ref{prop:TT}, one has
\[s:=\sum_{i=1}^r p^{m(d-m_i)}\leq p^{md},\]
where $d:=\max{m_i}$ and $p^m:=[k:k^p]$. 

Assume that $s=p^{md}$. Then one has 
\[p^{md}-1=s-1=\sum_{i=1}^r (p^{m(d-d_i)}-1) + (r-1).\]
This implies that $r-1=\dim G$ is divisible by $p^m-1$, and hence $\dim G\geq p^m-1$ (note that $G$ is nontrivial and connected, so $\dim G>0$). This contradicts to the assumption that $\dim G<p^m-1$. Therefore, $s<p^{md}$ and by Proposition \ref{prop:TT}, $H^1(k,G)$ is infinite.
\end{proof}

\subsection{Weil restriction}

To prove Theorem \ref{thm:main2}, we also need some basic facts about Weil restriction of linear algebraic groups over fields (equivalently, smooth affine group schemes over fields)  as presented in \cite[Appendices A.2-A.3]{Oe}.

Let $\rho:k\to k^\prime$ be a homomorphism of commutative rings, where $k^\prime$ is a projective $k$-module of finite type. For any affine $k$-scheme $W^\prime$ we then can associate an affine $k$-scheme $\sR_{k^\prime/k} W$ called the Weil restriction of $W$, which satisfies the following universal property: for any $k$-scheme $V$, one has a bijection (functorial in $V$)
\[\Hom_{k-sch}(V,\sR_{k^\prime/k}W)\to \Hom_{k^\prime-sch}(V\times_kk^\prime,W).\]
We refer the reader to the \cite[Chapter 7, 7.6]{BLR} for a more general study of Weil restriction.

\begin{lem}
 \label{lem:Weil restriction}
Let $\rho: k\to k^\prime$ be a finite field extension and $G^\prime$ be a linear algebraic group over $k^\prime$. Let $G=\sR_{k^\prime/k} G^\prime$ be its Weil restriction. The following properties are true.
\begin{enumerate}
 \item $G$ is a linear algebraic group and $H^1(k,G)\simeq H^1(k^\prime,G^\prime)$.
 \item $G$ is connected (resp. unipotent) if and only if  $G^\prime$ is connected (resp. unipotent).
 \item $G$ is unipotent and $k$-wound if and only if  $G^\prime$ is unipotent and $k^\prime$-wound.
 \item $G$ is unipotent and $k$-split if and only if  $G^\prime$ is unipotent and $k^\prime$-split.
\end{enumerate}
\end{lem}
\begin{proof}
 (1) These follow from \cite[Appendix 3, A.3.2]{Oe} and \cite[Chapter IV, 2.3, Corollary]{Oe}.

(2) This is \cite[Appendix 3, A.3.7]{Oe}.

(3) By the definition of Weil restriction, one has
\[
 \begin{aligned}
  G\left(k[[T]]\right)&\simeq G^\prime\left(k[[T]]\otimes_k k^\prime\right) \simeq G^\prime\left(k^\prime[[T]]\right),\\
  G\left(k((T))\right)&\simeq G^\prime\left(k((T))\otimes_k k^\prime\right)\simeq G^\prime\left(k^\prime((T))\right),
 \end{aligned}
\]
where $k[[T]]$, resp. $k^\prime[[T]]$,  is the ring of formal power series in one variable $T$ over $k$, resp. $k^\prime$ and  $k((T))$, resp.  $k^\prime((T))$, is the fraction of $k[[T]]$, resp. $k^\prime[[T]]$ and all  isomorphisms appeared are canonical. (Note that two canonical maps $k[[t]]\otimes_k k^\prime \simeq k^\prime[[t]]$ and $k((t))\otimes_k k^\prime \simeq k^\prime((t))$, $\left(\sum_i a_i t^i\right) \otimes\lambda\mapsto \sum_i \lambda a_i t^i$, are isomorphisms since $k^\prime/k$ is finite.)

By \cite[Chapter V.8, Proposition]{Oe}, for a unipotent algebraic group $U$ over a field $k$, $U$ is $k$-wound if and only if $U(k[[T]])=U(k((T)))$. The assertion then follows from this fact.

(4) First, assume that $G^\prime$ is $k^\prime$-split, we will show that $G$ is $k$-split by induction on $\dim G^\prime$. If $\dim G^\prime=1$, i.e., $G^\prime\simeq_{k^\prime} \G_a$, then $G$ is $k$-isomorphic to $\sR_{k^\prime/k} \G_a=\G_a^{[k^\prime:k]}$, which is $k$-split.

If $\dim G^\prime>1$ then there is a $k$-subgroup $H^\prime$ of $G^\prime$ such that $H^\prime$ is $k^\prime$-split and the quotient $G/H^\prime\simeq_{k^\prime} \G_a$. The exact sequence of $k^\prime$-groups
\[1\to H^\prime \to G^\prime \to \G_a\to 1\]
induces the following exact sequence of $k$-groups (\cite[Appendix 3, A.3.8]{Oe})
\[ 1\to \sR_{k^\prime/k}H^\prime \to \sR_{k^\prime/k} G^\prime=G\to \sR_{k^\prime/k}\G_a=\G_a^{[k^\prime:k]}\to 1.\]
From this exact sequence, we deduce that $G$ is $k$-split.

Second, assume that $G^\prime$ is not $k^\prime$-split we need to show that $G$ is not $k$-split. 
In fact, if $G^\prime$ is not connected then by (2) $G$ is not connected and hence $G$ is not $k$-split. 
We may assume that $G^\prime$ is connected. Let $G^\prime_s$ be the $k^\prime$-split of $G^\prime$. Then  $G^\prime_w:=G^\prime/G^\prime_s$ is $k^\prime$-wound of dimension $\geq 1$. 
We have the following exact sequence of $k^\prime$-groups
\[1\to G^\prime_s\to G^\prime \to G^\prime_w\to 1.\]
This exact sequence induces the following exact sequence of $k$-groups (\cite[Appendix 3, A.3.8]{Oe})
\[1\to \sR_{k^\prime/k}G^\prime_s\to G=\sR_{k^\prime/k}G^\prime \to \sR_{k^\prime/k}G^\prime_w\to 1.\]
As $\sR_{k^\prime/k}G^\prime_w$ is $k$-wound by (3) and of dimension $=[k^\prime:k]\dim G\geq 1$, it implies that $G$ is not $k$-split.
\end{proof}

\subsection{Special versus split unipotent algebraic groups}
\begin{defn}
 Let $k$ be a field, $G$ a smooth unipotent  algebraic $k$-group. We define the following two properties
\[ P(G;k)\hspace{2truecm} H^1(k,G)=0 \text{ if and only if } G \text{ is $k$-split}. \]
and
\[ SP(G;k) \hspace{2truecm} \text{$G$ is special if and only if $G$ is $k$-split}. \]
\end{defn}

\begin{rmks}
 (1) The property $\sP(G/k)$ does not always hold  in general, i.e., there is a field $k$ and a smooth unipotent algebraic $k$-group $G$ such that $H^1(k,G)=0$ but $G$ is not $k$-split.

(2) For any smooth algebraic unipotent $k$-group $G$, $P(G/k)$ implies evidently $SP(G/k)$.
\end{rmks}

Proposition \ref{prop:dim<p-1} can be restated as the following corollary.
\begin{cor}
 \label{cor:<p-1}
Let $k$ be a field of characteristic $p>0$, $v$ a non-trivial valuation of $k$. We assume that  $k$ has a finite $k^p$-valuation basis. Let $G$ a non-trivial smooth connected unipotent algebraic $k$-group of dimension $<p^m-1$. Then the property $P(G;k)$ holds.
\qed
\end{cor}

\begin{lem}
\label{lem:SP}
Let $k$, $K,L$ be fields such that $L/k$ is a (not necessarily algebraic) separable extension, $L/K$ is a finite extension. Let $G$ be a smooth unipotent algebraic $k$-group. Denote by $H$ the Weil restriction $\sR_{L/K}(G\times_k L)$. Then if $P(H;K)$ holds then $SP(G;k)$ holds.
\end{lem}
\begin{proof}
Assume that  $G$ is special, in particular, $H^1(L,G)=0$, we need to show that $G$ is $k$-split. By Lemma \ref{lem:Weil restriction} (1), $H^1(K,H)=H^1(L,G)=0$. Hence as $P(H;K)$ holds, $H$ is $K$-split. Also by Lemma \ref{lem:Weil restriction} (4), $G$ is $L$-split. Since $L/k$ is a separable extension, $G$ is also $k$-split by \cite[Chapter V.7, Proposition]{Oe}.
\end{proof}

\subsection{Proof of Theorem \ref{thm:main2}}
%\begin{proof}[Proof of Theorem \ref{thm:main2}]
If $k_0$ is perfect then $k$ is perfect and $G$ is always $k$-split and the assertion of the theorem holds trivially. 

From now on, we assume that $k_0$ is not perfect. In particular, it implies that the characteristic of $k_0$ is $p>0$. Note also that the valuation $v$ on $k_0$ is non-trivial since otherwise by Lemma \ref{lem:equality},  $[k_0:k_0^p]=1$, i.e., $k_0$ is perfect, a contradiction.

If $G$ is $k$-split then it is evident that $G$ is special. 

Assume now that $G$ is special, in particular connected. We take a natural number $m$ such that 
\[[k:k_0]\cdot\dim G <[k_0:k_0^p]p^m -1,\]
and choose $m$ variables $y_1,\ldots,y_m$ over $k$. 
We set 
\[L:=k(y_1,\ldots,y_m) \text{ and } K:=k_0(y_1,\ldots,y_m).\]
Then $L/k$ is a separable extension and $L/K$ is a finite extension. 
Denote by $H$ the Weil restriction  $\sR_{L/K}(G\times_k L)$. By Lemma \ref{lem:Weil restriction}, $H$ is a connected unipotent $K$-group with 
\[\dim H=[L:K]\dim (G\times_k L)\leq [k:k_0]\dim G< [k_0:k_0^p] p^m-1=[K:K^p]-1,\]
by the choice of $m$. (The last equality follows from \cite[Chapter V, 16.6, Corollary 3]{Bou1}.) Therefore, Proposition \ref{prop:dim<p-1} implies that $P(H;K)$ holds. Hence by Lemma \ref{lem:SP}, $G$ is $k$-split.
%\end{proof}
\qed

\subsection{Extension of valuations}
%To prove  Corollary \ref{cor:main3}, we need the following lemmas.
\begin{lem}
\label{lem:extension1}

Let $k$ be a field of characteristic $p>0$ with a valuation $v$, $\Gamma$ its value group. Let $K=k(x_1,\ldots,x_r)$ be the field of rational functions  in $r$ variables $x_1,\ldots,x_r$ with coefficients in $k$. 
Then there is a unique valuation $w$ on $K$ with value group $\Gamma\times \Z\times\cdots\times \Z$, $r$ times, (with lexicographical order from the right) such that $w(a)=(v(a),0,\ldots,0)$ for any $a\in k$ and $w(x_i)=(0,\ldots,1,\ldots,0)$, where $1$ is at the $i+1$-th position.

Furthermore, if $k$ has a finite $k^p$-valuation basis with respect to $v$ then $K$ has a finite $K^p$-valuation basis with respect to $w$. 

\end{lem}
\begin{proof}
For the first assertion, see \cite[Chapter VI, Section 10.3, Theorem 1]{Bou2}.

For the second assertion, by using induction on $r$, it suffices to consider the case $r=1$.

Let $n=[k:k^p]$ and let $(b_i)_i, 1\leq i\leq  n$ be a valuation basis with respect to $v$ of $k^p$-vector space $k$. Then we show that $(b_ix^j), 1\leq i\leq  n, 0\leq j\leq  p-1$, is a valuation basis with respect to $w$ of $k^p(x^p)$-vector space $k(x)$.

 Assume that $w(b_ix^j)\equiv w(b^kx^l)$ modulo $p$, or equivalently $(v(b_i),j)\equiv (v(b_k),l)$ modulo $p$. Hence $j\equiv l$ modulo $p$ and $v(b_i)\equiv v(b_k)$ modulo $p$. This implies that $j=l$ and $i=k$. Therefore $(b_ix^j), 1\leq i\leq  n, 0\leq j\leq  p-1$, are $k^p(x^p)$-valuation independent with respect to $w$.

It can be check that $[k(x):k^p(x^p)]=[k:k^p]\cdot p=np$. Hence $(b_ix^j), 1\leq i\leq  n, 0\leq j\leq  p-1$, is a $k^p(x^p)$-valuation basis of $k$ with respect to $w$.
\end{proof}

\begin{lem}
\label{lem:extension2}

Let $k$ be a field of characteristic $p>0$ with a valuation $v$, $\Gamma$ its value group. Let $K=k((x_1,\ldots,x_r))$ be the fraction field of the ring of formal power series  in $r$ variables $x_1,\ldots,x_r$ with coefficients in $k$. 
Then there is a  valuation $w$ on $K$ with value group $\Gamma\times \Z\times\cdots\times \Z$, $r$ times, (with lexicographical order from the right) such that $w(a)=(v(a),0,\ldots,0)$ for any $a\in k$ and $w(x_i)=(0,\ldots,1,\ldots,0)$, where $1$ is at the $i+1$-th position.

Furthermore, if $k$ has a finite $k^p$-valuation basis with respect to $v$ then $K$ has a finite $K^p$-valuation basis with respect to $w$. 
\end{lem}
\begin{proof} For simplicity of notation, we write a monomial $x_1^{n_1}\ldots x_r^{n_r}$ as $x^n$, interpreting $x$ as the vector $(x_1,\ldots,x_r)$ and $n$ as $(n_1,\ldots, n_r)$.

We define the map $w: k[[x_1,\ldots,x_r]]\to \Gamma\times\Z^r$ as follows. Define $w(0)=\infty$, and for each element $0\not=\sum_{n} a_nx^n\in k[[x_1,\ldots,x_r]]$, choose the smallest index $n_0$ such that $a_{n_0}\not=0$, and define  
\[w(\sum_{n} x^n):= (v(a_{n_0}),n_0).\]
Then $w$ is a valuation on $k[[x_1,\ldots,x_n]]$, i.e., $w$ satifies
\begin{enumerate}
\item $w(a+b)\geq \min\{w(a), w(b)\}$  for all $a,b\in k[[x_1,\ldots,x_r]]$,
\item $w(ab)=w(a)+w(b)$ for all $a,b\in k[[x_1,\ldots,x_r]]$, 
\item $w(0)=1$ and $w(0)=\infty$. 
\end{enumerate}
In fact, write $a=\sum_{n\geq n_0} a_n x^n$ with $a_{n_0}\not=0$ and $b=\sum_{m\geq m_0} b_{m} x^m$ with $b_{m_0}\not=0$, then we have
\[w(ab)=w\left(\sum_{n\geq n_0,m\geq m_0} a_nb_m x^{n+m}\right)
=(v(a_{n_0}b_{m_0}),n+m)=(a_{n_0},n)+(b_{m_0},m)=w(a)+w(b).\]
For (2), without the loss of generality we may assume that $n_0\leq m_0$ then $(v(a_{n_0}),n_0)\leq (v(b_{m_0}),m_0)$. If $n_0< m_0$ then 
\[v(a+b)=(v(a_{n_0}),n_0)=\min \{w(a),w(b)\}.\]
If $n_0=m_0$ then $v(a_{n_0})\leq v(b_{n_0})$ and
\[v(a+b)=(v(a_{n_0}+b_{n_0}),n_0)\geq (v(a_{n_0}),n_0)=\min \{w(a),w(b)\}.\]
Condition (3) is trivial.

By \cite[Chapter VI, Section 10, Proprosition 4]{Bou2}, we can extend uniquely $w$ to a valuation $w:K=k((x_1,\ldots,x_n))\to \Gamma\times\Z^r$.

For the last assertion, let $s=[k:k^p]$ and $b_1,\ldots,b_s$ is a $k^p$-valuation basis of $k$ then one can check that the values $v(b_i x_1^{n_1}\cdots x_r^{n_r})$, $1\leq i \leq s$, $0\leq n_1,\ldots,n_r\leq p-1$ are pairwise distinct modulo $p$. In particular, these elements $b_i x_1^{n_1}\cdots x_r^{n_r}$ are $k^p$-linearly independent. This implies that $[K:K^p]\geq sp^r$. 

On the other hand, $K=k((x_1,\ldots,x_r))$ is the completion of $L=k(x_1,\ldots,x_r)$ with respect to the valuation $w^\prime$ corresponding to $(x_1,\ldots,x_r)$ (note that in general $w^\prime$ is different from $w$ constructed as above). Then one has 
\[sp^r=[L:L^p]\geq [K:K^p],\]
the first equality follows from \cite[Chapter V, 16.6, Corollary 3]{Bou1} and the second inequality follows from \cite[Lemma 2.1.2]{GO}. Therefore, $[K:K^p]=sp^r$ and  the elements $b_i x_1^{n_1}\cdots x_r^{n_r}$, $1\leq i \leq s$, $0\leq n_1,\ldots,n_r\leq p-1$, form a $K^p$-valuation basis of $K$.
\end{proof}

\subsection{Geometric fields and Corollary \ref{cor:main3}}

Lemma \ref{lem:extension1} and Lemma \ref{lem:extension2} motivate the following definition.
\begin{defn}
\label{defn:geometric}
 Let $k\subset K$ be two fields. We say that $K$ is {\it geometric} over $k$ if there is a tower of finite length of field extensions 
\[K=K_0\supset K_1\supset K_2 \supset \cdots \supset K_n=k\]
such that $K_0\supset K_1$ is a finite field extension and for each $i\geq 1$, we have 
\begin{enumerate}
 \item $K_i=K_{i+1}(x_1,\ldots,x_r)$ for some variables $x_1,\ldots,x_{r_i}$ or
 \item $K_i=K_{i+1}((y_1,\ldots,y_r))$ for some variables $y_1,\ldots,y_{s_i}$.
\end{enumerate}
\end{defn}

\begin{cor}
\label{cor:geometric}
 Let $K$ be a field which is geometric over a perfect field $k$. Let $G$ be a non-trivial smooth unipotent algebraic $K$-group. Then $\ed_K(G)=0$ if and only if $G$ is $K$-split.
\end{cor}
\begin{proof}
By assumption there is a tower of finite length of field extensions 
\[K=K_0\supset K_1\supset K_2 \supset \cdots \supset K_n=k\]
as in Definition \ref{defn:geometric}. If $K_1=K_n=k$ then $K$ is perfect. 
 The corollary then holds trivially.

Now assume that $K_1\not=K_n $. On $K_n=k$ we consider the trivial valuation $w_n$. Then since $K_n$ is perfect, $K_n$ has a finite $K_n^p$-valuation basis with respect to $w_n$, namely $\{1\}$. Therefore by Lemma \ref{lem:extension1} and Lemma \ref{lem:extension2}, the valuation $w_n$ extend to a \emph{non-trivial} valuation $v$ on $K_0$ so that $K_0$ has a finite $K_0^p$-valuation basis with respect to $v$. The corollary now follows from Theorem \ref{thm:main2}.
\end{proof}
Corollary \ref{cor:main3} is just a very special case of Corollary \ref{cor:geometric}. 

\subsection{Unipotent algebraic groups of dimension one}
Over fields which are geometric over a perfect field, we can compute the essential dimension of smooth unipotent algebraic group of dimension $1$ as follow.
\begin{prop}
 Let $k$ be a field which is geometric over a perfect field. Let $G$ be a smooth connected unipotent algebraic $k$-group of dimension $1$. Then $\ed_k(G)=0$ if $G$ is $k$-isomorphic to $\G_a$ and $\ed_k(G)=1$ otherwise.
\end{prop}
\begin{proof}
 If $G\simeq_k \G_a$ then it is trivial that $\ed_k(G)=0$.

Assume now that $G$ is not $k$-isomorphic to $\G_a$, i.e., $G$ is not $k$-split. In particular, it implies that $k$ is not perfect and hence infinite.  It is well-known that $G$ is commutative and annihilated by $p$ ($G$ is in fact a $k$-form of $\G_a$). Therefore, by Lemma \ref{lem:1}, $\ed_k(G)\leq 1$. 

On the other hand, by Corollary \ref{cor:geometric}, $\ed_k(G)\leq 1$ since $G$ is not $k$-split. Therefore, $\ed_k(G)=1$.
\end{proof}


\begin{thebibliography}{9999999}
%\bibitem[BT]{BT} L. Bary-Soroker and N. D. Tan, On $p$-embedding problems in characteristic $p$, arXiv:1008.1879.

\bibitem[BF]{BF} G. Berhuy and G. Favi, \emph{Essential dimension: A functorial point of view (after A. Merkujev)}, Doc. Math. {\bf 8} (2003) 279-330.
\bibitem[Bo]{Bo} A. Borel, \emph{Linear Algebraic Groups (2nd ed.)}, Graduate texts in mathematics {\bf 126}, New York: Springer-Verlag 1991.

\bibitem[BLR]{BLR} S. Bosch, W, L\"{u}tkebohmert and M. Raynaud, \emph{N\'eron models}. Ergebnisse der Mathematik und ihrer Grenzgebiete (3) [Results in Mathematics and Related Areas (3)], 21. Springer-Verlag, Berlin, 1990.

\bibitem[Bou1]{Bou1} N.~Bourbaki, {\it Elements of Mathematics: Algebra II, Chapters 4-7}, (translated by P. M. Cohn and J. Howie), Springer-Verlag 1990.

\bibitem[Bou2]{Bou2} N.~Bourbaki, {\it \'El\'ements de Math\'ematics: Alg\`ebra commutative, Chapitres 5 \`a 7}. Springer-Verlag Berlin Heidelberg 2006.

\bibitem[BR]{BR} J. Buhler and Z. Reichstein, {\it On the essential dimension of a finite group}, Compositio Math. {\bf 106} (1997), no. 2, 159-179.

\bibitem[BRV1]{BRV1} P. Brosnan, Z. Reichstein and A. Vistoli, {\it Essential dimension and algebraic stacks}, arXiv:math/0701903.


\bibitem[BRV2]{BRV2} P. Brosnan, Z. Reichstein and A. Vistoli, {\it Essential dimension of moduli of curves and other algebraic stacks}, (with an appendix by N. Fakhruddin), arXiv:0907.0924, to appear in Journal of European Mathematical Society.

\bibitem[CGP]{CGP} B. Conrad, O. Gabber and G. Prasad, {\it Pseudo-reductive groups}, Series: New Mathematical Monographs (No. 17) (to appear).

\bibitem[DK]{DK} L. van den Dries and F. Kuhlmann, {\it Images of additive polynomials in ${\mathbb F}_q((t))$ have the optimal approximation property}, Canad. Math. Bull. {\bf 45} (2002), no. 1, 71-79.

\bibitem[Fl]{Fl} M. Florence, {\it On the essential dimension of cyclic $p$-groups}, Invent. Math. {\bf 171} (2007), 175-189.

%\bibitem[Ga]{Ga} S. Garibaldi

\bibitem[GO]{GO} O. Gabber and F. Orgogozo, {\it Sur la $p$-dimension des corps}, Invent. Math. {\bf 174} (2008), 47-80.

\bibitem[GM]{GM} P. Gille and L. Moret-Bailly, {\it Actions alg\'ebriques de groupes arithm\'etiques}, appendice \`a l'article de Ullmo-Yafaev "Galois orbits and equidistribution of special subvarieties: towards the Andr\'e-Oort conjecture'', pr\'epublication.

%\bibitem[Gir]{Gir} J. Giraud, {\it Cohomologie non ab\'elienne}, Springer-Verlag, Die Grundlehren der mathematischen Wissenschaften, Band {\bf 179},  Berlin, 1971.

\bibitem[Gro]{Gro} A. Grothendieck, {\it Torsion homologique et sections rationnelles}, Anneaux de
Chow et Applications, S\'eminaire Claude Chevalley, 1958, expos\'e n. 5.

\bibitem[JLY]{JLY} C. U. Jensen, A. Ledet and N. Yui, {\it Generic polynomials: Constructive aspects of the inverse Galois problem}, Mathematical Sciences Research Institute Publications {\bf 45}, Cambridge University Press, 2002.

\bibitem[KM]{KM} N. Karpenko, A. Merkurjev, {\it Essential dimension of finite $p$-groups}, Invent. Math. {\bf 172} (2008), 491-508. 		 

\bibitem[Le]{Le} A. Ledet, {\it On the essential dimension of $p$-groups}, Galois Theory and Modular Forms, Dev. Math., vol. {\bf 11}, Kluwer Acad. Publ., 2004, pp. 159-172.

\bibitem[LMMR]{LMMR} R. L\"otscher, A. Meyer, M. MacDonald, Z. Reichstein, {\it Essential $p$-dimension of algebraic tori}, preprint 2009.

\bibitem[Me]{Me} A. Merkujev, {\it Essential dimension}. Quadratic forms - algebra, arithmetic, and geometry, 299-325, Contemp. Math., {\bf 493}, Amer. Math. Soc., Providence, RI, 2009.

\bibitem[MZ]{MZ} A. Meyer, Z. Reichstein, {\it Some consequences of the Karpenko-Merkurjev theorem},  to appear in the issue of Documenta Math.  dedicated to Andrei Suslin's 60th birthday.

\bibitem[Oe]{Oe} J. Oesterl\'e, {\it Nombre  de Tamagawa et groupes unipotents en  caract\'eristique $p$},  Invent. Math. {\bf 78} (1984), 13-88.

\bibitem[Re]{Re} Z.~Reichstein, {\it Essential dimension}, to appear in Proceedings of the International Congress of Mathematicians 2010.

%\bibitem[Ru]{Ru} P. Russell, Forms of the affine line and its additive group, Pacific J. Math. {\bf 32} (1970) 527-539.
\bibitem[Se1]{Se1} J.-P. Serre, {\it Espaces fibr\'es alg\'ebriques}, Anneaux de Chow et Applications, S\'eminaire Claude Chevalley, 1958, expos\'e n. 1.

\bibitem[Se2]{Se2} J.-P. Serre, {\it Galois cohomology}, Corr. 2 printing; Springer 2002 (Springer Monographs in Mathematics).

%\bibitem[TT1]{TT1} Nguyen Q. Thang and Nguyen D. Tan, {\it On the surjectivity of the localization maps for Galois cohomology of algebraic groups over fields}, Comm. Algebra {\bf 32} (2004), 3169-3177.  
\bibitem[TT1]{TT1} Nguyen Q. Thang and Nguyen D. Tan, {\it On the Galois and flat cohomology of unipotent algebraic groups over local and global function fields. I},  J. Algebra  {\bf 319}  (2008),  no. 10, 4288-4324.

\bibitem[TT2]{TT2} Nguyen D. Tan and Nguyen Q. Thang, {\it Galois cohomology of unipotent algebraic groups and field extensions}, preprint (to appear in Comm. Algebra).

\bibitem[TV]{TV} D. Tossici and A. Vistoli, {\it On the essential dimension of infinitesimal group schemes}, arXiv:1001.3988 (to appear in  Amer. J. Math.).
\end{thebibliography}
\end{document}